\documentclass[a4paper,12pt]{article}

\usepackage[utf8]{inputenc}
\usepackage[english]{babel}
\usepackage{amssymb,amsmath,amsthm,tikz}
\usepackage{graphicx}

\newcommand{\A}{\mathbf{A}}
\newcommand{\B}{\mathbf{B}}
\newcommand{\C}{\mathbf{C}}
\newcommand{\D}{\mathbf{D}}

\newcommand{\RR}{\mathbb{R}}

\newcommand{\E}{\mathbb{E}}
\newcommand{\EE}{\mathcal{E}}

\newcommand{\QQ}{\mathbb{Q}}

\newcommand{\I}{\mathbf{I}}
\newcommand{\II}{\mathcal{I}}
\newcommand{\J}{\mathbf{J}}
\newcommand{\K}{\mathbf{K}}
\newcommand{\KK}{\mathcal{K}}

\newcommand{\tr}{\textnormal{tr}}
\newcommand{\Aut}{\textnormal{Aut}}
\newcommand{\Age}{\textnormal{Age}}
\newcommand{\Fra}{\text{Fra\"iss\'e}}

\theoremstyle{definition}
\newtheorem{theorem}{Theorem}[section]
\newtheorem{definition}[theorem]{Definition}
\newtheorem{prop}[theorem]{Proposition}
\newtheorem{corollary}[theorem]{Corollary}
\newtheorem{lemma}[theorem]{Lemma}

\newtheorem{fact}[theorem]{Fact}
\newtheorem{question}{Question}

\title{Box Ramsey and Canonical Colourings}
\author{Keegan Dasilva Barbosa}

\begin{document}
	
	\maketitle

	\begin{abstract}
		This paper introduces the concept of a productive notion of big Ramsey degree and showcases its versatility through a handful of applications. The main focus is notably providing sufficient conditions for the existence of a finite canonical basis of equivalence relations, building upon the prior work of Laflamme, Sauer, and Vuksanovic. Additionally, a combinatorial analysis of indexed structures is conducted.
	\end{abstract}

	\section{Introduction}
	Structural Ramsey theory has garnered a significant amount of attention ever since the advent of the KPT correspondence, a beautiful theorem that marries finite combinatorics and topological dynamics. Much of the discussion in Structure Ramsey theory has centered around relational expansions, as they are the premier tool for computing Ramsey degrees \cite{ CDP1, CDP2, Dobrinen, ExactBRD, NVT,ProdRamsey,Zucker1,Zucker2,Zucker3}. The reason being is quite simple. Rather than work with a structure directly, it is easier to encode the structure onto something where a Ramsey type statement is well known (often a tree), and reduce the problem to simply counting the number of expansions which is a simple exercise in combinatorics.\\
	\\
	Not all Ramsey phenomena are equal however. In the case of \textit{small} Ramsey degrees, where quantifiers only occur on finite structures, Ramsey degrees are finite if and only if they are encoded by a relational expansion \cite{KPT, NVT, Zucker1}. Moreover, the existence of finite small Ramsey degrees in a \Fra\ class corresponds to the metrizability of the universal minimal flow of $\textnormal{Aut}(\K)$, where $\K$ is the associate \Fra\ structure. It is unclear however if \textit{big} Ramsey degrees are also always characterized by relational expansions. Zucker was able to show that if big Ramsey degrees for a structure could be given by a relational expansion, then there is an associated dynamical assertion about the universal completion flow associated to $\textnormal{Aut}(\K)$. Given relational expansions are the premier technology to compute Ramsey degrees, it seems the field does not quite know where to begin with this problem. \\
	\\
	This ultimately leads to the question we aim to solve in this paper. Rather than try to tackle Zucker's question directly, we look for a side channel and establish a connection between big Ramsey degrees and \textit{canonical} expansions. Canonical expansions are by no means a novel concept. Ramsey, the quintessential pioneer of the field, was the first to successfully solve the canonical expansion problem for $\omega$. Erd\"os and Rado used Ramsey's theorem to explicitly formulate precisely what the canonical equivalence relations associated to $\omega$ must be \cite{ErdosRado}. 
	In this paper, we extend the scope of Ramsey's groundbreaking result to encompass the following generalization.
	
	\begin{theorem}
		Let suppose $\mathcal{L}$ be a finite relational language with no unaries. Let $\K$ be an infinite $\mathcal{L}$ structure. If every substructure of cardinality $d$ has finite big Ramsey degree in $\K$, then there is a finite list of $d$-ary relations on $\K$ up to restriction.  
	\end{theorem}
	
	This will appear in this paper as Theorem \ref{GeneralizedER}. A corollary to the above is that there is also a finite list of equivalence relations. Of all the relations to consider on a structure $\K$, equivalence relations are the most important from the perspective of Ramsey theory, as they are in one-to-one correspondence with colourings modulo permutations of colour classes.
	
	\begin{theorem}
		Let $\K$ be a relational structure with the property that for every $n$, there are only finitely many substructures of cardinality $n$. Every finite substructure has finite big Ramsey degree if and only if every finite substructure admits a finite \textit{Ramsey basis}.
	\end{theorem}
	A \textit{Ramsey basis} is simply a set of equivalence relations that is maximal with respect to restriction to subcopies of $\K$.
	We will also see that if a structure $\K$ is fairly saturated, namely if it satisfies a condition we call \textit{free point duplication}, we can characterize its equivalence relations in a very strong sense.
	\begin{theorem}
		Suppose $\K$ is a countable \Fra\ structure that satisfies free point duplication. For any equivalence relation $\E$ on $\K$, there is a $\K^\prime \in \binom{\K}{\K}$ such that $\E$ restricted to $\K^\prime$ is either the equality relation, or the trivial relation with one class. Equivalently, for any $\chi:\K \rightarrow \omega$, there is $\K^\prime \in \binom{\K}{\K}$ such that $\chi \upharpoonright \K^\prime$ is either injective or constant. 
	\end{theorem}
	The above appears as Theorem \ref{equivdichotomy}. We are also able compute big Ramsey degrees for indexed structures as defined by Kubis and Shelah \cite{Diversifications}, provided that the indexed set and labeled set are equivalent. We then use this to get a small application to computing big Ramsey degrees for a very specific class of Urysohn spaces.\\
	\\
	All of our results are a consequence of a productive version of Ramsey degree that we denote \textit{box Ramsey degree}. The most important aspect of this combinatorial characteristic is that it is equivalent to the existence of finite big Ramsey degrees provided we are working in a relational language with at most finitely many relations of each arity. \\
	\\
	The paper is structured as follows: Section 2 serves as a preliminary, where we define the essential objects in structural Ramsey theory. Additionally, we provide a brief overview of topological dynamics, and demonstrate that $\QQ$ holds a universal property as a \Fra\ object. Section 3 will encompass our main results, prominently featuring the box Ramsey theorem along with its various applications. In Section 4, we present our concluding remarks and discuss open questions that arise from our findings. Following that, in Section 5, we express our heartfelt gratitude to the remarkable individuals who have influenced the author throughout this journey. We also extend our appreciation to the Fields Institute for their generous support and funding.

	\section{Preliminaries}

	\subsection{Structural Ramsey Theory}
	
	At its core, structural Ramsey theory is the study of regularity preservation under partition. The basic unit of measurement in the partition calculus is that of \textit{Ramsey degree}, an integer representing how close an object is to satisfying Ramsey's theorem in the classical sense.

	\begin{theorem}[Ramsey's Theorem (finitary) \cite{Ramsey}]
		For every collection of positive integers $m,n,k \in \omega$ with $m<n$, there is an integer $N\in \omega$ such that for every for every $k$-colouring $c:[N]^m \rightarrow k$, there is $S\subseteq N$ of size $n$ such that $c$ is constant on $[S]^m$.  
	\end{theorem}

	\begin{theorem}[Ramsey's Theorem (infinitary) \cite{Ramsey}]
		For every pair of positive integers $m,k\in \omega$ and colouring $c:[\omega]^m \rightarrow k$, there is an infinite $S\subseteq \omega$ such that $c$ is constant on $[S]^m$.
	\end{theorem}

	In the structural setting, we replace integers with \textit{structures}. Most of our focus will be on relational structures, though some of our results can likely be brought to a more general setting. 
	
	\begin{definition}
		A (relational) \textit{structure} $ \A$ is a pair $(A, \{R_i^\A \}_{i\in I })$ where $A$ is a set and $R_i^\A$ are relations of finite arity on $A$.  
	\end{definition}

	\begin{definition}
		Suppose $\A$ and $\B$ are structures in the same relational signature $\{R_i \}_{i\in I}$. An \textit{embedding} of $\A$ into $\B$ is an injection  $f:A\rightarrow B $ such that for all $i\in I$, $R_i^\A(a_1,...,a_{n_i})\iff R_i^\B(f(a_1),...,f(a_{n_i}))$. We call the image of $f$ a \textit{copy} of $\A$ in $\B$. We let $\binom{\B}{\A}$ denote the set of copies of $\A$ in $\B$. 
	\end{definition}

	It is important to emphasize the if and only if in the definition of embedding. Consider the class of graphs as an example. What is colloquially referred to as subgraph, is not equivalent to a copy in our language. Instead, copies are those subgraphs which are \textit{induced}.

	\begin{definition}
		Given a (possibly infinite) relational structure $\K$, we let $\Age(\K)$ denote the class of all finite substructures of $\K$. 
	\end{definition}

	\begin{definition}
		Let $\KK$ be a class of finite structures in some relational signature. We say $\A$ has \textit{Ramsey degree} $t\in \omega$ if for every $\B \in \KK$ and $k\in \omega$, there is a $\C\in \KK$ such that for every $c:\binom{\C}{\A}\rightarrow k$, there is a $\B^\prime \in \binom{\C}{\B}$ such that $|c[\binom{\B^\prime }{\A}]| \leq t$. If no such $t$ exists, we say $\A$ does not have finite Ramsey degrees. 
	\end{definition}

	\begin{definition}
		We call a class of finite structures a \textit{Ramsey class}, or simply just \textit{Ramsey}, if every $\A\in \KK$ has finite Ramsey degree $1$. 
	\end{definition}

	Many classes are Ramsey. The most prototypical example is the class of finite linear orders, which is an immediate consequence of Ramsey's theorem. A majority of modern structural Ramsey theory centers around \textit{\Fra}\ classes. The reasoning is two fold. On the one hand, classes that are sufficiently directed and Ramsey are also \Fra.\ On the other hand, a remarkable result by Kechris, Pestov, and Todorcevic correlates the Ramsey property to \textit{extreme amenability}.

	\begin{definition}
		We call a topological group $G$ \textit{extremely amenable} if every continuous action of $G$ on a compact space $X$ admits a fixed point. 
	\end{definition}

	\begin{definition}
		We call a class of finite relational structures $\KK$ \Fra\ if and only if:
		\begin{itemize}
			\item If $\B\in \KK$ and $\A$ is an induced substructure of $\B$, $\A\in \KK$ (the hereditary property).
			\item For every $\A,\B \in \KK$, there is a $\C\in \KK$ which embeds both $\A$ and $\B$ (the joint embedding property).
			\item For every $\A,\B,\C \in \KK$ and pair of embeddings $f_1:\A\rightarrow \B$ and $g_1:\A\rightarrow \C$, there is a $\D\in \KK$ and embeddings $f_2:\B\rightarrow \D$ and $g_2 : \C\rightarrow \D$ such that $f_2\circ f_1 = g_2 \circ g_1$ (the amalgamation property).
		\end{itemize}
	\end{definition}

	\begin{theorem}[\Fra\ \cite{Fraisse}]
		$\KK$ is a \Fra\ class if and only if $\KK = \Age(\K)$ for some countable ultrahomogeneous structure $\K$. 
	\end{theorem}

	\begin{theorem}[Kechris, Pestov, Todorcevic \cite{KPT}]
		Suppose $\K$ is a \Fra\ structure, and let $\KK = \Age(\K)$. The following are equivalent:
		\begin{itemize}
			\item $\KK$ is a Ramsey class of rigid structures. 
			16gLaJKcDHVHbccbo9ayPsopB6gA1pZB888m8q6g2T8x6kBm
		\end{itemize}
	\end{theorem}

	A consequence of the above has been a stream of research equating dynamical properties to Ramsey theoretic phenomena \cite{NVT,Zucker1,Zucker2,Zucker3}. This connection between finite combinatorics and topological dynamics is not inherently new, with topological variants of both Van der Waerden's theorem, and Szemer\'edi's theorem arising in the late 70's and early 80's \cite{Furst,ErgodicSz}. What was novel however, was how many examples of extremely amenable automorphism groups existed. \\
	\\
	We will close this subsection with a brief discussion on \textit{big Ramsey degrees}. Big Ramsey degrees, as the name may suggest, are vastly more challenging to compute than their small counterparts. Much of the modern focus is on the computation of big Ramsey degrees of \Fra\ structures \cite{Dobrinen,ExactBRD,Sauer,Zucker2,Zucker3}.

	\begin{definition}
		Suppose $\K$ is an infinite structure and let $\A\in \Age(\K)$. The \textit{big Ramsey degree} of $\A$ in $\K$, if it exists, is the smallest integer $t$ such that for every $k\in \omega$ and $c: \binom{\K}{\A}\rightarrow k$, there is $\K^\prime \in \binom{\K}{\K}$ such that $|c[\binom{\K^\prime}{\K}]| \leq t$. We denote this integer $t(\A,\K)$.
	\end{definition}

	A question posed by Ma\v{s}ulovi\'c at the AMS special session on Ramsey theory of infinite structures was whether or not big Ramsey degrees are reasonably productive. This question is motivated by the small setting, where the Ramsey degree of a product is precisely the product of Ramsey degrees \cite{ProdRamsey}. It is clear that this fails in the big setting, but it was unclear whether or not a product Ramsey degree could be reasonably computed. In Section 3.1, we will give the precise necessary and sufficient conditions required for a product Ramsey degree (which we denote \textit{box Ramsey degree}) to exist and give an explicit computation.

	\subsection{Brief Primer on Dynamics}
	We will give a fairly brief primer on dynamics. A majority of the paper will be devoted to combinatorial methods of problem solving, though dynamical consequences are extremely important to structural Ramsey theory, and this always deserves recognition. This subsection will mostly be for non-professionals in the field, and can mostly be skipped. Material here will be at its most relevant in Section 2.3 and Section 3.4.

	\begin{definition}
		We let $\mathbf{S}_\infty$ denote the group of permutations $f:\omega \rightarrow \omega$ endowed with the topology of pointwise convergence. 
	\end{definition}

	The above space is \textit{Polish}, and is foundational to the study of \Fra\ structures. This is due to the one-to-one correspondence between closed subgroups of $\mathbf{S}_\infty$ and \Fra\ structures. For more on the model theoretic or descriptive set theoretic background, the author recommends the following texts by Hodges \cite{Hodges} and Kechris \cite{ClassDesc} respectively.

	\begin{fact}
		Given a closed $G\leq \mathbf{S}_\infty$, there is a relational structure $M=(\omega, \{R_i\}_{i\in I})$ such that  $\Aut(M) =G$ and $M$ is \Fra. 
	\end{fact}
	The automorphism group of any relational structure where relations have finite arity can be identified with a closed subgroup of $\mathbf{S}_\infty$, hence \Fra\ structures and subgroups of $\mathbf{S}_\infty$ are inextricably linked. For more on this correspondence, see Hodges \cite{Hodges}. The key ingredient in this correspondence is \textit{type}.

	\begin{definition}
		Suppose $G\leq \mathbf{S}_\infty$ is closed. A \textit{type} $\tau \subseteq [\omega]^n $ is a $G$-orbit. We call $\tau$ an $n$-\textit{type}. 
	\end{definition}

	\begin{definition}
		We say $G\leq \mathbf{S}_\infty$ is \textit{oligomorphic} if for every $n\in \omega$, there are only finitely many $n$-types. 
	\end{definition}

	We will predominantly be interested in structures with oligomorphic isomorphism group, or more explicitly, with only finitely many substructures of any designated finite cardinality. Our inference being that this is slightly stronger than the existence of finitely many \textit{traces}, a concept we introduce in Section 3.1.

	\subsection{Universality of $\QQ$}
	
	We finish this section with a small exercise. We will provide a short proof that the rationals are ubiquitous amongst \Fra\ structures with extremely amenable automorphism groups. This has always been implicitly understood, though we will prove it unequivocally. Doing so provides a partial answer to a question posed by Valery Michkine to the author at the Toronto Set Theory seminar.
	
	\begin{theorem}[Kechris-Pestov-Todorcevic]
		Let $G\leq \mathbf{S}_\infty $ be a closed subgroup. If $G$ is extremely amenable, then there is a linear ordering $<_G$ on $\omega$ that is $G$-invariant.
	\end{theorem}

	\begin{proof}
		The space $\textnormal{LO}(\omega) $ of linear orders on $\omega$ is compact. The group $G$ naturally acts on $\textnormal{LO}(\omega)$ via the logic action $g\cdot \preceq$, where $x (g\cdot\preceq) y$ if and only if $g^{-1}(x) \preceq g^{-1}(y)$. This action is continuous, and consequently, there is a $<_G \in \textnormal{LO}(\omega)$ which is $G$-fixed. That is, for every $g\in G$, if $x<_G y$, then $g(x) <_G g(y)$. 
	\end{proof}

	For more details on the above, see \cite{KPT}.The above is simply a fragment of the entire Kechris-Pestov-Todorcevic correspondence, though it will be sufficient for our purposes.

	\begin{definition}
		Given a closed extremely amenable $G\leq \mathbf{S}_\infty$, we let $<_G$ denote a $G$ invariant order on $\omega$. 
	\end{definition}
	
	Jahel has provided excellent work in his thesis, whereby an analysis of linear orders and $2$-types could be used to deduce intriguing dynamical properties \cite{Jahel}. We will also take this approach to prove a more combinatorial result. 
	
	\begin{lemma}\label{typepushing}
		Fix a $G\leq \mathbf{S}_\infty$ closed and extremely amenable, and suppose $G$ acts transitively on $\omega$. Fix $\tau\subseteq [\omega]^2$ a $2$-type. For all $x\in \omega$, there is a $y\in \omega$ such that $x<_G y$ and $\{x,y\} \in \tau$. 
	\end{lemma}

	\begin{proof}
		Fix $\tau$ and $x\in \omega$. Take $F\in \tau$. Since $G$ acts transitively on $\omega$, there is a $g\in G$ such that $g(x) \in F$. It follows that $x\in g^{-1}(F) \in \tau$. So, suppose $\{x,y\} \in \tau$. If $x<_G y$, we are done. Suppose instead that $y<_G x$. Find $h\in G$ such that $h(y) =x$. It follows that $x<_G h(y)$, and $\{x,h(y)\} = h(\{y,x \}) \in \tau$ as desired. 
	\end{proof}

	\begin{theorem}
		Fix a $G\leq \mathbf{S}_\infty$ closed and extremely amenable, and suppose $G$ acts transitively on $\omega$. The ordering $<_G$ is a dense linear order without endpoints. 
	\end{theorem}

	\begin{proof}
		Suppose $x\in \omega$ is $<_G$ maximal, and take $y\in \omega$. Since $G$ acts transitively on $\omega$, there is a $g\in G$ such that $g(y)=x$. Since $y<_Gx$, $x<_G g(x)$ which contradicts the maximality of $x$. Showing $<_G$ is unbounded below is symmetric. We now show density. Take $x,z\in \omega$ with $x<_G z$. Let $\tau$ be the type of $\{x,z\}$. By Lemma \ref{typepushing}, there is a $ z^\prime$ such that $z<_G z^\prime$ and $\{z,z^\prime\} \in \tau $. Since $\{x,z\}, \{z,z^\prime\} \in \tau$, we can find a $g\in G$ for which $g(z)=x$ and $g(z^\prime)=z$. Since $x<_Gz <_G z^\prime$, it follows that $x<_G g(z)<_G z$. Hence, $<_G$ is dense. See Figure 1. 
		
		\begin{figure}[htp]
			\centering 
			\begin{tikzpicture}[main/.style = {fill, circle}] 
				\coordinate (X) at (-3,0);
				\coordinate (Z) at (0,0);
				\coordinate (Zprime) at (3,0);
				\node [label=$x$] at (X) {$\bullet$};
				\node [label=$z$] at (Z) {$\bullet$};
				\node [label=$z^\prime$] at (Zprime) {$\bullet$};

				\draw (-5,0) -- (5,0);

				\coordinate (X) at (-3,-2);
				\coordinate (Z) at (-1.5,-2);
				\coordinate (Zprime) at (0,-2);
				\node [label=$x$] at (X) {$\bullet$};
				\node [label=$g(z)$] at (Z) {$\bullet$};
				\node [label=$z$] at (Zprime) {$\bullet$};

				\draw (-5,-2) -- (5,-2);

			\end{tikzpicture} 
			\caption{Density achieved by squishing $z$ with automorphism.}
		\end{figure}
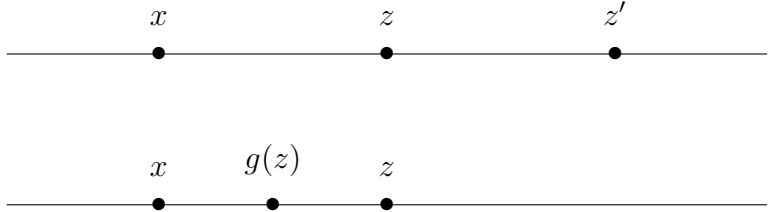
		
	\end{proof}

	\begin{definition}
		Let $(L,<_L)$ be a linear order. Suppose $\mathcal{I}$ is a family of disjoint convex subsets of $L$. The \textit{block order on} $\mathcal{I}$ is the order given by $I<J$ if and only if $\forall x\in I$ $\forall y\in J$, $x<_L y$. 
	\end{definition}

	\begin{corollary}
		Fix a $G\leq \mathbf{S}_\infty$ closed and extremely amenable, and suppose $G$ acts transitively on $\omega$. It follows that $G\leq \Aut(\QQ) $.
	\end{corollary}

	\begin{proof}
		Since $(\omega, <_G)$ is isomorphic to $\QQ$, without loss of generality, we may identify $(\omega, <_G)$ with $(\QQ,<)$. Since $G$ respects the ordering of the rationals, $G\leq \Aut(\QQ)$.  
	\end{proof}

	\begin{theorem}
		Fix a $G\leq \mathbf{S}_\infty$ closed and extremely amenable, and suppose there are $\kappa$-many $1$-types, for some countable $\kappa$. Let $\omega = \bigcup\limits_{i\in \kappa} P_i$ be the partition associated to type. $G\leq \Aut(\QQ) $.
	\end{theorem}

	\begin{proof}
		Since $G$ is extremely amenable, $|P_i|$ is either infinite or a singleton for all $i\in \kappa$. We may suppose they are all infinite without loss of generality. Let us consider a variant of a condensation argument penned by Laver \cite{Laver}. We construct an equivalence relation $\sim$ on $\omega$ defined by the following criterion:
		\begin{itemize}
			\item For all $x,y\in \omega$, if $x\sim y$, then $y\sim x$.
			16gLaJKcDHVHbccbo9ayPsopB6gA1pZB888m8q6g2T8x6kBm
			\item For all $x,y \in \omega$, if $\{b\in \omega : x <_G b <_G y\}$ is a dense linear order, then $x\sim y$.
		\end{itemize}
		It is easy to verify that $\sim$ is an equivalence relation, and the $\sim$ equivalence classes are convex eg. if $x\sim y$ and $x<_G b <_G y$, then $x\sim b$. \\
		\\
		\textbf{Claim:} For all $i\in \kappa$, $P_i$ belongs to exactly one $\sim$-class. Fix $i\in \kappa$ and take $x,y \in P_i$. Suppose that $x<_G y$. By Lemma \ref{typepushing}, and since $G$ acts transitively on $P_i$, there is a $z\in P_i$ such that $y<_G z$ and the pairs $\{x,y\}$ and $ \{y,z\}$ have the same type. Letting $g\in G$ be such that $g(y)=x$ and $g(z)=y$, we have that since $x<_G y<_G z$, $x<_G g(y) <_G y$ as in the previous theorem. Thus, $g(y)$ is the witness to density of the interval $\{b\in P_i : x<_G b <_G y \}$. It follows then that $x\sim y$, and hence, every $x,y\in P_i$ is in the same $\sim$ class. \\
		\\
		Our previous claim has shown that the partition $\{P_i : i\in \kappa \} $ is finer than $\omega/\sim$. Hence, classes in $\omega/\sim$ are $G$-invariant dense linear orders without endpoints. Let $\mathcal{I}$ be a partition of $\QQ$ into open bounded intervals. This partition inherits a natural linear order from $\QQ$, and in fact is isomorphic to $\QQ$, with the block order. Since $\QQ$ is universal over countable linear order, we can find $\mathcal{J} \subseteq \mathcal{I} $ and an indexing $\{L_I: I\in \mathcal{J}  \}$ of $\omega/\sim$ with the property that for all $I<_\mathcal{I} J$, if $x\in L_I$ and $y\in L_J $, $x<_G y$. Since each $L_J \in \omega/\sim$ and $J\in \mathcal{J} $ is a dense linear order without endpoints, there is an embedding $f: (\omega,<_G) \rightarrow \QQ$ with the property that $f[L_J] = J$. We know that $G\leq \Aut(\omega,<_G) $, and $f$ witnesses a group isomorphism from $\Aut(\omega,<_G) $ to the subgroup $\{g\in \Aut(\QQ) : \forall I \in \mathcal{I}\setminus \mathcal{J}, g\upharpoonright I = \text{id}_I  \} \leq \Aut(\QQ)$. Thus, $G\leq \Aut(\QQ)$.
	\end{proof}

	\section{Main Results}

	\subsection{Box Ramsey Degrees}

	For this section, fix an infinite (not necessarily countable nor ultrahomogeneous) structure $\K$. We discuss a multidimensional Ramsey property we call the \textit{box Ramsey degree}. A key part of computing the combinatorial characteristic will be projecting the problem to one of a lower dimension. We aptly call our projection operation the \textit{trace}.

	\begin{definition}
		Fix $d\in \omega$, $\A_1,...,\A_d \in \textnormal{Age}(\K)$. Take $(x_1,...,x_d )\in \ \prod\limits_{i=1}^d \binom{\K}{\A_i} $. We let $\tr(x_1,...,x_d) =\bigcup\limits_{i=1}^d x_i $.
	\end{definition}

	\begin{definition}
		Fix $d\in \omega$, $\A_1,...,\A_d \in \textnormal{Age}(\K)$. Take $(x_1,...,x_d ),(y_1,...,y_d) \in  \prod\limits_{i=1}^d \binom{\K}{\A_i} $. We say $(x_1,...,x_d)\simeq (y_1,...,y_d)$ if there is a partial automorphism $f:\K \rightarrow \K$ such that $\textnormal{dom}(f)=\tr(x_1,...,x_d) $, $ \textnormal{ran}(f) =\tr(y_1,...,y_d) $, and $f[x_i] = y_i $. 
	\end{definition}

	In the case that $\K$ is ultrahomogeneous, $\simeq$-equivalence can be witnessed by a full automorphism opposed to a partial automorphism. It can be easily verified that $\simeq$ is an equivalence relation. The relation $\simeq$ is a generalized version of the relation found in \cite{CanPart}. We can iterate a more general form of the Laflamme, Sauer, Vuksanovic argument to deduce the existence of \textit{box Ramsey degrees}, which we can now define.

	\begin{definition}
		Given $\A_1,...,\A_d \in \textnormal{Age}(\K) $, we let $t_\square((\A_1,...,\A_d), \K)$ denote the smallest integer $t$, if it exists, such that for every colouring \\$c: \prod\limits_{i=1}^d \binom{\K}{\A_i} \rightarrow k$, there is a $\K^\prime \in \binom{\K}{\K}$ such that $\vert c[ \prod\limits_{i=1}^d \binom{\K^\prime}{\A_i}] \vert \leq t$. We call $t$ the \textit{box Ramsey degree} of the tuple $(\A_1,...,\A_d)$. 
	\end{definition}

	We have chosen to call this the \textit{box} Ramsey degree because we are only allowed to make one selection of $\K\in \binom{\K}{\K}$. Hence, in the case that we have $\A_1=...=\A_d$, we are trying to minimize the amount of colour classes a colouring $c$ of $\binom{\K}{\A_1}^d$ meets on a subcube $\binom{\K^\prime}{\A_1}^d$. 
	
	\begin{lemma}
		The $\simeq$-classes are \textit{persistent} in $ \prod\limits_{i=1}^d \binom{\K}{\A_i} $ i.e for every $\K^\prime \in \binom{\K}{\K}$ and $P \in  \prod\limits_{i=1}^d \binom{\K}{\A_i}/\simeq$, $P\cap  \prod\limits_{i=1}^d \binom{\K^\prime}{\A_i} \neq\emptyset $.
	\end{lemma}

	\begin{proof}
		Fix a $\simeq$ equivalence class $P$, and set $\B = \tr(x_1,...,x_d)$ for some $(x_1,...,x_d) \in P$. Take $\K^\prime \in \binom{\K}{\K}$. It is clear that $\binom{\K^\prime}{\B}$ is nonempty, so take $\B^\prime \in \binom{\K^\prime}{\B}$. We can find a decomposition $\B^\prime = \bigcup\limits_{i=1}^d y_i$ such that $(y_1,...,y_d)\simeq (x_1,...,x_d)$. It follows immediately that $(y_1,...,y_d) \in  \prod\limits_{i=1}^d \binom{\K^\prime}{\A_i} $.
	\end{proof}
	
	Persistent sets are foundational to understanding lowerbounds of Ramsey degrees, as they describe to us which patterns are essential/unavoidable. Hence, we immediately have the following. 
	
	\begin{corollary}
		For every $\A_1,...,\A_d \in \textnormal{Age}(\K)$, we have the lower bound $t_d((\A_1,...,\A_d) ,\K) \geq \vert \prod\limits_{i=1}^d \binom{\K^\prime}{\A_i}/\simeq \vert $. 
	\end{corollary}

	The $\simeq$-classes are not the only obstruction to finite box Ramsey degrees, we also need to worry about the Ramsey degree of traces. We will show in the ultimate theorem of this subsection that these two obstructions are in fact the only obstructions to box Ramsey degrees.

	\begin{lemma}\label{TraceLemma}
		Fix a $\simeq$ equivalence class $P$, and set $\B = \tr(x_1,...,x_d)$ for some $(x_1,...,x_d) \in P$. Suppose $t(\B,\K) =t \in \omega$. Then for every colouring $c:P\rightarrow k$, there is a $\K^\prime \in \binom{\K}{\K}$ such that $\vert c[P\cap \prod\limits_{i=1}^d \binom{\K^\prime}{\A_i}] \vert \leq t\cdot \vert \textnormal{Aut}(\B) \vert$.
	\end{lemma}

	\begin{proof}
		Notice that on the class $P$, the mapping $\tr: P \rightarrow \binom{\K}{\B}$ is an $\textnormal{Aut}(\B)$-to-one map. To see this, take $(x_1,...,x_d)\simeq (y_1,...,y_d)$ and suppose $\tr(x_1,...,x_d) = \tr(y_1,...,y_d) = \B$. Since $(x_1,...,x_d) \simeq (y_1,...,y_d)$, there is an automorphism $f\in \textnormal{Aut}(\B)$ such that $f[x_i] = y_i$. So $\tr$ has at most $\textnormal{Aut}(\B)$ many pre-images. Moreover, there are exactly $\textnormal{Aut}(\B)$ many preimages. This is because for every automorphism $f$ and tuple $(x_1,...,x_d)$, $(x_1,...,x_d) \simeq (f[x_1],...,f[x_d] ) $. Consider then the natural action of $\textnormal{Aut}(\B)$ on $P$. The trace operation is bijective modulo this equivalence. Take a colouring $c: P \rightarrow k$ and consider the colouring $\hat{c}: \binom{\K}{\B} \rightarrow 2^k$ given by 
		\begin{align*}
			\hat{c}(\B^\prime) = \{c(y_1,...,y_d): \tr(y_1,...,y_d)=\B^\prime \}
		\end{align*}
		By the way we've defined our colouring, for every $\B^\prime \in \binom{\K}{\B} $, we have the upper bound $\vert\hat{c}(\B^\prime) \vert \leq \vert \textnormal{Aut}(\B) \vert$. Given $t$ is the Ramsey degree of $\B$ in $\K$, we can find a copy $\K^\prime \in \binom{\K}{\K}$ such that $\vert \hat{c}[\binom{\K^\prime}{\B}] \vert \leq t$. Given that for every $(y_1,...,y_d) \in P \cap \prod\limits_{i=1}^d \binom{\K^\prime}{\A_i}$, $c(y_1,...,y_d) \in \hat{c}(\tr(y_1,...,y_d )) $, we have achieved the bound $\vert c[P\cap \prod\limits_{i=1}^d \binom{\K^\prime}{\A_i}] \vert \leq t\cdot \vert \textnormal{Aut}(\B) \vert $ as desired. 
	\end{proof}

	It is clear to see from the above that the upperbound $t\vert \textnormal{Aut}(\B)\vert $ is tight. This leads to our main result of this subsection.

	\begin{theorem}
		Take $\A_1,...,\A_d \in \textnormal{Age}(\K) $. Let $\vec{x}_i $, $i\in \{1,...,m\} $, be a $\simeq$-transversal, and let $\B_i =\tr(\vec{x}_i)$. We then have the explicit equation for the box Ramsey degree $t_d((\A_1,...,\A_d),\K) = \sum\limits_{i=1}^m t(\B_i,\K) \cdot \vert\textnormal{Aut}(\B_i) \vert$.
	\end{theorem}

	\begin{proof}
		Let $P_1,...,P_m$ be an exhaustion of the equivalence classes $\prod\limits_{i=1}^d \binom{\K}{\A_i}/\simeq$ with $\tr(\vec{x})=\B_i$ for all $\vec{x} \in P_i $. Take a colouring $c:\prod\limits_{i=1}^d \binom{\K^\prime}{\A_i} \rightarrow k$. Set $\K_0 = \K $ and iteratively apply Lemma \ref{TraceLemma} to find a $\K_{i} \in \binom{\K_{i-1}}{\K}$ such that $c$ takes at most $t(\B_i,\K) \cdot \vert\textnormal{Aut}(\B_i) \vert$ many colours on $P_i \cap \prod\limits_{i=1}^d \binom{\K_i}{\A_i} $. It immediately follows that $c$ takes at most $\sum\limits_{i=1}^m t(\B_i,\K) \cdot \vert\textnormal{Aut}(\B_i) \vert$ many colours on $\prod\limits_{i=1}^d \binom{\K_m}{\A_i} $. That this quantity is a lower bound can be seen by constructing a bad colouring on each of the $P_i$ with disjoint ranges. 
	\end{proof}
	
	We now move to some applications of the existence of box Ramsey degrees. We will start by showing a canonical colouring theorem in the next subsection. 
	
	\subsection{Canonical Relations}
	Similar to big Ramsey degrees, canonical classifications of relations like the Erd\"os-Rado theorem are hard to come by. Some work has been done recently in the structural setting, with the biggest result being that of Laflamme, Vuksanovic, and Sauer, extending the Erd\"os-Rado theorem to countable universal binary homogeneous structures \cite{CanPart}. In this subsection, we extend this to an even wider class of infinite structures, and consider all possible relations. 
	\\
	\\
	We will start this subsection by defining the space of equivalence relations. We have two reasons for focusing on equivalence relations instead of colourings, which we've been using up to this point. The primary reason is that there is no hope in finding a finite list of representatives for $\omega$-colourings $\chi $, because every injection $f:\omega\rightarrow \omega$ encodes a distinct colouring $f\circ \chi$. Equivalence relations are blind to the names of colour classes, and hence there is more hope for categorizing them as it has been historically done (see previous paragraph). The secondary reason is dynamical, though we leave this analysis open-ended. The space of $\omega$-colourings is not compact. However, the space of equivalence relations is. Moreover, this compact space can be naturally acted on via the logic action. Hence, the space of equivalence relations are more amenable to dynamical applications than the space of $\omega$-colourings. 
	
	\begin{definition}
		Let $\K$ be a structure and let $\A\in \textnormal{Age}(\K) $. We let $\EE_\K(\A) $ denote the space of equivalence relations on $\binom{\K}{\A}$, where the topology is inherited from $2^{\binom{\K}{\A}^2 } $.
	\end{definition}

	\begin{fact}
		If $\K$ is a countable structure, $\EE_\K(\A)$ is compact. 
	\end{fact}

	\begin{definition}
		Given $\E \in \EE_\K(\A) $ and $\K^\prime \in \binom{\K}{\K} $, we let $\E \upharpoonright \K^\prime$ denote the restriction of $\E$ to $\binom{\K^\prime}{\A}$.
	\end{definition}

	\begin{definition}
		We call $\mathcal{B} \subseteq \EE_\K(\A) $ a \textit{Ramsey basis} if for every $\E \in \EE_\K(\A) $, there is a $\tilde{\E} \in \mathcal{B} $ and $\K^\prime \in \binom{\K}{\K}$ such that $\E \upharpoonright \K^\prime = \tilde{\E} \upharpoonright\K^\prime $
	\end{definition}
	
	A trick often implemented in structural Ramsey theory is to construct new colourings from old ones. One common such example is that of the product colouring, where we take two colourings $\delta_i:\binom{\K}{\A} \rightarrow k_i$ $i\in \{1,2\}$, and reproduce the new colouring $(\delta_1,\delta_2):\binom{\K}{\A} \rightarrow k_1 \times k_2$. There are many instances where this trick is used, but the author recommends \cite{Zucker2} for some examples. It seems that implimenting this trick is all we need to deduce a Ramsey basis from finite box Ramsey degree. 
	
	\begin{theorem}\label{GeneralizedER}
		Fix an $\A \in \textnormal{Age}(\K)$. If $t_\square((\A,\A),\K)<\infty$, there is a finite Ramsey basis $\mathcal{B} \subseteq \EE_\K(\A)$. 
	\end{theorem}
	\begin{proof}
		Take an equivalence relation $\E \in \EE_\K(\A)$. Let $t=t_\square((\A,\A),\K) $, and let $c:\binom{\K}{\A}^2 \rightarrow t$ be a bad colouring. Let $\delta: \binom{\K}{\A}^2 \rightarrow \{0,1\}$ be defined by $\delta(\A_1,\A_2) = 1$ if and only if $\A_1 \E \A_2$. Consider the product colouring $c_\E:\binom{\K}{\A}^2 \rightarrow \{0,1\}\times t$ given by $c_\E(\A_1,\A_2) = (\delta(\A_1,\A_2), c(\A_1,\A_2)  )$. As $t_\square((\A,\A),\K)$ is finite and $c$ is bad, there is a $\K^\prime \in \binom{\K}{\K}$ and an $S\subseteq t$ such that for all $(\A_1,\A_2) \in \binom{\K^\prime}{\A}$, $\A_1 \E \A_2 $ if and only if $c(\A_1,\A_2)\in S$. It follows then that the finite set $\mathcal{B}$ of equivalence relations $\hat{\E}$ of the form $\A_1 \hat{\E}\A_2$ if and only if $c(\A_1,\A_2) \in S$ for appropriately chosen $S\subseteq t$ forms a Ramsey basis. 
	\end{proof}
	As an immediate consequence, we get Theorem 1.2. However, the above does not tell us how to find the smallest such finite $\mathcal{B}$, nor does it tell us the cardinality of the $B$ we found. At best, it tells us that $\vert \mathcal{B} \vert \leq 2^{t_2((\A,\A),\K)}$. This upperbound is very crude, especially when one considers the Erd\"os-Rado canonical partition theorem on $\omega$. However, in some special cases, we can reduce the number of representatives provided we know in advanced some properties of the structure $\K$. We can also make a statement about relations more generally.
	
	\begin{theorem}
		Let $\A$ be a singleton, and suppose $\L$ is in a language with no unaries so that without loss of generality, $\binom{\K}{\A} = \K $. If $t_\square((\A_1,...,\A_d),\K) < \infty$ where $\A= \A_i$ for all $i\in \{1,...,d\}$, then there is a list of $2^{t_\square((\A_1,...,\A_d),\K)}$ many $d$-ary relations on $\K$ up to restriction.  
	\end{theorem}
	\begin{proof}
		Let $R$ be a $d$-ary relation on $\K$. Similar to the proof of the previous theorem, we can view the indicator of $R$, $\delta: \K^d \rightarrow \{0,1\} $ $\delta(x_1,...,x_d) \iff R(x_1,...,x_d)$, as a colouring. Applying the same argument as before, we take a bad colouring $\chi:\K^d \rightarrow t_\square((\A_1,...,\A_d),\K)$, and apply the product colouring argument to find a collection of colour classes $S\subseteq t_\square((\A_1,...,\A_d),\K)$ and $\K^\prime \in \binom{\K}{\K}$ such that $R(x_1,...,x_d) \iff \chi(x_1,...,x_d)\in S$ for all $(x_1,...,x_d) \in (\K^\prime)^d$.   
	\end{proof}
	We then get a precise computation of the number of relations that exist on $\omega$ up to restriction, getting a precise number to Ramsey's theorem of relations on $\omega$ \cite{Ramsey}. This can also be slightly modified to compute the explicit number of canonical relations on $\QQ$.
	\begin{corollary}
		There is a family of precisely $2^{\sum\limits_{k=1}^d s(d,k) } $ many relations on $\omega$ up to restriction, where $s(d,k) = \sum\limits_{i=0}^k (-1)^i \binom{k}{i}(k-i)^d $, the number of surjections from a set of size $d$ to a set of size $k$. 
	\end{corollary}
	\begin{proof}
		Let $\A_1=...=\A_d$ be the isomorphism class of a singleton in $\omega$. Since singletons have Ramsey degree $1$ in $\omega$, by the pigeonhole principle, it suffices to compute the box Ramsey degree by counting the number of $\simeq$ classes in $\omega^d$. Take $\vec{x}= (x_1,...,x_d) \in \omega^d$. Note that the trace $\text{tr}(\vec{x})$ is isomorphic as a linear order to $\{1,...,k\}$. Let $g_{\vec{x}}:\tr(\vec{x}) \rightarrow \{1,...,k\}$ be the unique such isomorphism. Consider $f_{\vec{x}} :\{1,...,d\}\rightarrow \{1,..,k\}$ be given by $f_{\vec{x}}(i) = j$ if and only if $g_{\vec{x}}(x_i) = j$. Given $\vec{x} \simeq (f_{\vec{x}}(1),..., f_{\vec{x}}(d) )$, it is easy then to verify that $\vec{x} \simeq \vec{y}$ if and only if $f_{\vec{x}} = f_{\vec{y}}$. Consequently, every $\simeq$ class is encoded by a surjection from $\{1,...,d\}$ to $\{1,...,k\}$ for some $k\in \{1,...,d\}$. Hence, $t_\square((\A_1,...,\A_d),\omega) = \sum\limits_{k=1}^d s(d,k) $ and we're done by the previous theorem. 
	\end{proof}
	
	\subsection{Analysis For Saturated Structures}
	Understanding the history of the big Ramsey degree computation for $\QQ$ is quite insightful for better understanding the decisions and directory the field of structural Ramsey theory is taking. First, in an unpublished note, Laver shows that the Ramsey degrees for $\QQ$ were finite by an application of Milliken's theorem for trees \cite{Devlin}. The trick is to encode the rationals onto the infinite binary tree in such a manner that the tree order encapsulates information about the order on $\QQ$. Getting an exact computation required Devlin to understand the fundamental patterns that would appear in the tree via antichains, now given his namesake \textit{Devlin types}. In a paper by Sauer, encoding structures onto trees is extended to universal countable binary structures, like the Random graph \cite{Sauer}. The most sophisticated extension of this are \textit{coding tree} constructions of Dobrinen \cite{Dobrinen}. The core idea, which we will further elaborate on below, is as follows:
	\begin{itemize}
		\item Enumerate your structure $\K$.
		\item Create a tree $T$ with the property that the $i$th level has a unique representative for the $i$th indexed member from $\K$, and passing type encapsulates information on how the $i$th member relates to previous members. 
		\item Prove a Ramsey theorem specialized to the tree $T$, and use that there are only finitely many ways to represent a structure $\A$ in $T$ to deduce an upperbound on $t(\A,\K)$. 
	\end{itemize}
	The hardest step is by far the last step, where the current popular methodology for complex structures is an appeal to forcing \cite{CDP1,CDP2,Dobrinen,Zucker2}. Interestingly, strong amalgamation conditions correspond to a type of duplication property for nodes in $T$. This duplication can then be used to explicitly compute big Ramsey degrees by counting antichains, akin to Devlin \cite{CDP1,CDP2}. We will show that a special duplication property can lead to efficient computations of Ramsey bases when $\A$ is a singleton. We fix a finite language $\mathcal{L} = \{R_i : i \in L \}$ of irreflexive binary relations with associated \Fra\ class $\KK$, \Fra\ limit $\K$, and enumeration $\K = \{v_i : i \in \omega   \} $. It is common to assume without loss of generality that for every distinct pair $x,y\in \K$, there is a unique $i\in L$ such that $x R_i y $. We will move forward with this assumption. 
	
	\begin{definition}
		Consider the tree of finite partial functions into $L$, $L^{<\omega}$, made into an $L$-ary tree via the extension relation $\sqsubseteq$. For each $i \in \omega$, let $c_i\in L^{<\omega}$ denote the unique node of length $i$ with the property that $c_i(k) = j $ if and only if $v_k R_j v_i$. We refer to $C$ as \textit{coding nodes}. Let $T$ denote the $\sqsubseteq$-downward closure of $C$. We call $T$ the \textit{coding tree} associated to $\K$ (and the provided enumeration). 
	\end{definition}
	\begin{definition}
		Given a finite antichain $S=\{s_i:i<n\} \subseteq T$ where each node appears on a different level, we denote \textit{the structure encoded by} $S$ to be the isomorphism type of a structure $\mathbf{D}=\{d_i: i<n\} $ where $d_k R_j d_i$ if and only if $s_i (|s_k|) = j$, where $|s_k|$ is the height of $s_k$. 
	\end{definition}
	
	Note, the definition we have chosen to use is entirely combinatorial compared to the more model theoretic definition found in the works of Coulson, Dobrinen, and Patel \cite{CDP1,CDP2} and is more in line with the definition from Dobrinen's initial work on the Henson graph \cite{Dobrinen}. The model theoretic framework is arguably a better one, as it allows for an analysis where the binary relations are not irreflexive. In particular, there is a strong classification of the big Ramsey degrees of $\QQ_\QQ$ in \cite{CDP1} that our methods indexed here will not work on, as we'd need to consider separating the reflexive components of the binary relations into infinitely many unaries, which makes the method of analyzing trees fall apart. Despite this, this will not be a concern to us. In fact, the ultimate result of this section fails in the case of $\QQ_\QQ$, precisely because the structure itself is equipped with a nontrivial equivalence relation that must be preserved when moving to a subcopy of the rationals. 
	
	\begin{definition}[Coulson, Dobrinen, Patel \cite{CDP1,CDP2}]
		Let $T$ be a coding tree of some structure $\K$. Let $\A$ be a finite substructure of $\K$ with cardinality $n$. Let $\mathcal{OA}$ denote the set consisting of one representative from each isomorphism class of ordered copies of $\A$. Given $(\A,<)\in \mathcal{OA}$ ($\A = \{a_i: i<n\} $) we say a tree $S\subseteq T$ is a \textit{diagonal tree coding} $(\A,<)$ if the following hold:
		\begin{enumerate}
			\item $S$ is a finite tree with $n$ terminal nodes and branching degree two.
			\item $S$ has at most one branching node in any given level, and no two distinct nodes from among the branching nodes and terminal nodes have the same length. Hence, $T$ has $2n-1$ many levels. 
			\item Let $\{d_i: i <n\}$ be an enumeration of the terminal nodes of $S$. The increasing bijection from $\{a_i: i <n\} $ to $\{d_i: i <n\} $ is an orer isomorphism between $\A$ and the structure encoded by $\{d_i: i <n\}$. 
		\end{enumerate}
		We let $\mathcal{DA}(\A,<)$ denote the set of isomorphism classes of distinct diagonal trees encoding $(\A,<)$. 
	\end{definition}
	Note our third condition differs slightly from what can be found in the Coulson, Dobrinen, Patel paper, but is identical given our encoding. 
	\begin{theorem}[Coulson, Dobrinen, Patel \cite{CDP1,CDP2}]\label{CDPTheorem}
		Suppose $\mathcal{L}$ is a language consisting of finitely many binary relations. Suppose $\K$ is \Fra\ with \Fra\ limit $\K$, and satisfies SDAP\textsuperscript{+} or LSDAP\textsuperscript{+}. Fix $\A\in \KK$. We then have the following equation
		\begin{align*}
			t(\A,\K) &= \sum\limits_{(\A,< )\in \mathcal{OA}} |\mathcal{DA}(\A,<)|
		\end{align*}
	\end{theorem}
	
	\begin{definition}
		We say a structure $\K$ has \textit{free duplication} if for every $j\in L$, $x,y\in \K $ such that $x R_i y$, there are infinitely many $z\in \K $ such that $xR_i z$ and $z R_j y $
	\end{definition}
	Note that $\QQ$, the Random graph, and the $k$ edge labeled Random graph are examples of such structures. Free duplication is a saturation type condition that translates to combinatorial properties of the coding tree associated to $\K$. 
	\begin{theorem}\label{equivdichotomy}
		Let $\E$ be an equivalence relation on $\K$. There is a copy $\K^\prime \in \binom{\K}{\K}$ such that $\E\upharpoonright \K^\prime$ is either the equality relation, or $\E$ only has one class.
	\end{theorem}
	\begin{proof}
		Since Ramsey degrees are finite by Theorem \ref{CDPTheorem}, it suffices to assume $\E$ is of the form 
		\begin{align*}
			x\E y &\iff \{x,y\} \in S\\
			S &\subseteq \bigcup\limits_{\substack{ \A\in \KK \\ |\A|\leq 2} } \bigcup\limits_{(\A,<)\in \mathcal{OA}} \mathcal{D}(\A,<)
		\end{align*}
		We break then into cases. There are only two to consider. If the set $S$ contains only trees of height $1$, $x\E y$ if and only if $x=y$. This is because if $x\neq y$, $(\{x.y\},<)$ is encoded by trees of height $3$. So, we assume the set $S$ contains a tree $D$ of height $3$. Without loss of generality, let us assume that $x$ and $y$ are encoded by the leaves of $D$ with $x<y$, $x$ of higher height than $y$, and $yR_i x$. We may also suppose that there is some other coding node $p$ between the root of $D$ and $x$.\\
		\\
		Consider some $R_j$. By point duplication, we know for certain that there is some $z$ such that $xR_j z$ and $y R_j z$. Applying point duplication to $x$ again if need be, we may assume instead that the height of $x$ is larger than the height of $z$ and $z R_j x $. It follows that the antichain $\{z,y\}$ encodes a tree isomorphic to $D$, but $\{x,z\}$ (up to increasing the height of $x$) can be of the form $x<z$, $\text{ht}(x)<\text{ht}(z)$, and the passing is of type $j$, or $x<z$, $\text{ht}(x)>\text{ht}(z)$, and the passing is of height $j$. See Figure 1.
		\\
		\\
		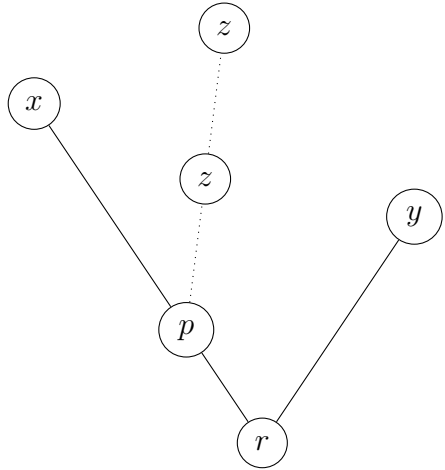
\begin{figure}[htp]
			\centering 
			\begin{tikzpicture}[main/.style = {draw, circle}] 
				\node[main] at (0,0) (1) {$r$};
				\node[main] at (2,3) (2) {$y$};
				\node[main] at (-1,1.5) (3) {$p$};
				\node[main] at (-3,4.5) (4) {$x$};
				\node[main] at (-.75,3.5) (5) {$z$};
				\node[main] at (-.5,5.5) (6) {$z$};
				\draw (1) -- (2);
				\draw (1) -- (3);
				\draw (3) -- (4);
				\draw[dotted] (3) -- (5);
				\draw[dotted] (5) -- (6);
			\end{tikzpicture} 
			\caption{Two possible $z$ with same passing relative to $y$ as $x$, but can have distinct passing and height with respect to $x$.}
			\label{fig:my_label}
		\end{figure}
		Since $\E$ is an equivalence relation, $x\E y$, and $y\E z$, it follows that $x\E z$ and so the other two trees where the passing is $j$ must appear in $S$. Since $j$ was arbitrary, all trees of height $3$ (with $2$ antichains) must belong to $S$ and hence $\E$ has only one equivalence class. It follows that $\E$ was either constant or equality as desired. 
	\end{proof}
	A simple consequence is the \textit{rainbow Ramsey} property \cite{Rainbow} for singletons. 
	\begin{corollary}
		If $\chi : \K \rightarrow \omega$ is such that for each $i\in \omega$, $|\chi^{-1}(i)|$ finite, these is $\K^\prime \in \binom{\K}{\K}$ such that $\chi^\prime \upharpoonright \K^\prime$ is injective. 
	\end{corollary}

	\subsection{Application to Indexed Structures}

	In this subsection, we will analyze Ramsey properties of structures indexed by other structures. Fix countable structures $\K$ and $\I$ and set $\KK= \textnormal{Age}(\K) $, and $\II = \textnormal{Age}(\I) $. We will view $\mathcal{I}$ as a class of indices, and $\mathcal{K}$ as a class of labels. Labeled structures have a deep history in combinatorics and have found a variety of applications. For example, Laver was able to encode linear orderings onto labeled trees in order to prove a conjecture by Fra\"iss\'e \cite{Laver}. Labeled structures were also more recently considered by Kubis and Shelah in order to create a \Fra\ structure with automorphism group non-universal over the automorphism groups of its substructures \cite{Diversifications}. We will show the machinery of box Ramsey degrees can be used to compute big Ramsey degrees in the case that $\KK=\II$, though the question is still open in the general setting.

	\begin{definition}
		An $\II$ indexed $\KK$ structure is a pair $(\J,l)$ where $\J \in \II$ and $l:\J\rightarrow \KK$. We call $l$ a labeling. We let $\KK^\II$ denote the class of all $\II$ indexed $\KK$ structures. This class forms a category with arrows of the form $(\iota,  \vec{f}): (\mathbf{J},l) \rightarrow (\mathbf{H},q)$, where $\iota: \J \rightarrow \mathbf{H}$, and $\vec{f}$ is indexed by $J$ with $f_j\in \textnormal{\textbf{Emb}}(l(j), q(\iota(j)) )$. Composition is given by $(\iota,\vec{f})\circ(\jmath, \vec{h}) = (\jmath \circ \iota, (f_{\iota(j) }\circ h_j)_{j\in\mathbf{J} } )$.
	\end{definition}

	Substructures of a $\KK^\II$-structure $ (\mathbf{H},q)$ are then parameterized by some $\J\subseteq \mathbf{H}$ and $l(j)\subseteq q(j) $ for each $j\in \J$. In the case that $\I$ is the countable structure with no relations, $\KK^\II$ is precisely \textit{the diversification of} $\KK$ as defined by Kubis and Shelah \cite{Diversifications}.

	\begin{definition}
		Given structures $\J$ and $\A$, we let $(\J,l_\A)$ denote the indexed structure who's labeling is $l_\A$ is the constant $\A$. 
	\end{definition}

	\begin{prop}
		If $\KK$ and $\II$ are Fra\"iss\'e, then so is $\KK^\II$. 
	\end{prop}

	This is simple to verify. The Fra\"iss\'e limit of $\KK^\II$ is simply $(\I,l_\K) $. There is an intriguing relation between $\Aut((\I,l_\K))$ and the groups $\Aut(\K)$ and $\Aut(\I)$.

	\begin{prop}
		$\Aut((\I,l_\K)) = \Aut(\K) \wr_{\I} \Aut(\I)$. 
	\end{prop}

	\begin{proof}
		Recall, the definition of the wreath product $\Aut(\K) \wr_{\I} \Aut(\I)$ is the semidirect product $\Aut(\K)^\I \rtimes \Aut(\I)$, where $\Aut(\I)$ acts on $\Aut(\K)^\I$ the natural way, $\iota \cdot (g_i)_{i\in \I} = (g_{\iota(i)})_{i \in \I} $. Consider the computation of a product of pairs 
		\begin{align*}
			((g_i)_{i\in \I}, \iota) \cdot ((h_i)_{i\in \I}, \jmath) &= (((g_i)_{i\in \I}\cdot \iota) (h_i)_{i\in \I}, \iota \circ \jmath)\\
			&= ( (g_{\iota(i) h_i})_{i\in \I}, \iota \circ \jmath  )
		\end{align*}
		Given how we have defined composition for indexed structures, one immediately has that the mapping $\phi:\Aut(\K) \wr_{\I} \Aut(\I) \rightarrow \Aut( (\I,l_\K))$ given by $\phi((\iota, \vec{g} )) = ((g_i)_{i\in \I}, \iota) $ is a group isomorphism. 
	\end{proof}

	\begin{lemma}\label{Cofinal}
		If $\KK$ satisfies JEP, $\{(\J, l_\A): \J \in \II, \A \in \KK \}$ is a cofinal subclass of $\KK^\II$. 
	\end{lemma}

	\begin{proof}
		Take $(\J,l) \in \KK^\II$. Since $\J$ is finite and $\KK$ has JEP, there is an $\A\in \KK$ such that for all $j\in \J$, there is an embedding $f_j :l(j) \rightarrow \A$. It follows then that $(\text{id}_\J, f_j) : (\J,l) \rightarrow (\J, l_\A)$ is an embedding.
	\end{proof}

	\begin{theorem}[Jahel and Zucker \cite{JahelZucker}]
		If $1\rightarrow H \rightarrow G \rightarrow K\rightarrow 1$ is a short exact sequence of polish groups, and the universal minimal flows $M(H)$ and $M(K)$ are metrizable, then so is $M(G)$. Moreover, if $H$ and $K$ are extremely amenable, so is $G$. 
	\end{theorem}

	\begin{corollary}
		If $\KK$ and $\II$ have finite Ramsey degrees, then so does $\KK^\II$. Moreover, if both $\KK$ and $\II$ have the Ramsey property, so does $\KK^\II$. 
	\end{corollary}

	We now move to the study of big Ramsey degrees for indexed structures. Here, things become more challenging, as there need not be a connection between a structure $\K$ and its indices $\I$. In the case that $\K=\I$, we have a partial answer.

	\begin{theorem}
		Suppose every sequence in $\KK$ has finite box Ramsey degree in $\K$. Take $(\A_0, l) \in \KK^\KK$. Assume that the universe of $\A_0$ is $\{1,...,n\}$ and let $l(i) = \A_i$. The big Ramsey degree of $(\A_0, l)$ is bounded above by $t_\square( (\A_0,...,\A_n),\K)$.      
	\end{theorem}

	\begin{proof}
		Take a colouring $c: \binom{(\K,l_\K )}{(\A_0,l)} \rightarrow k$. We can identify each $(\A_0^\prime, l^\prime) \in \binom{(\K,l_\K )}{(\A_0,l)}$ uniquely with a sequence $(\A_0^\prime, \A_1^\prime,..., \A_n^\prime) \in \prod\limits_{i=0}^n \binom{\K}{\A_i}$. Since $t_\square ( (\A_0,...,\A_n),\K)$ is finite, we can find $\K^\prime \in \binom{\K}{\K}$ such that $c$ takes at most $t_\square ( (\A_0,...,\A_n),\K)$ many values restricted to $(\K^\prime, l_{\K^\prime})$. 
	\end{proof}

	While quite specialized, the above can be used to efficiently compute upperbounds for big Ramsey degrees of a very niche class of metric structures. It is quite likely this enumeration technique can be modified to a broader scope of metric \textit{spectra}, though we leave this for a later date.

	\begin{definition}
		Take $S\subseteq \RR^+ $ finite. We call such $S$ \textit{spectra}. We let $\mathcal{M}_S$ denote the category of finite metric spaces with distances in $0\cup S$, where arrows are isometric embeddings.   
	\end{definition}
	By considering isometric embeddings as arrows, as is usually done, we can view $\mathcal{M}_S$ as a family of relational structures, with $|S|$ symmetric binary relations $\{ R_s:s\in S\}$ where $xR_s y \iff d(x,y) =s$.
	\begin{definition}
		Given a metric spectra $S\subseteq \RR^+$, a \textit{block decomposition} of $S$ is a partition $S=\bigcup\limits_{i=1}^k B_i$ recursively constructed as follows:
		\begin{itemize}
			\item $B_1 = \{a\in S: a \leq 2\text{min}S \}$
			\item $B_{i+1} = \{a \in S\setminus \bigcup\limits_{j=1}^i B_i: a \leq 2 \text{min} (S\setminus \bigcup\limits_{j=1}^i B_i) \} $
		\end{itemize}
	\end{definition}
	Note that $ \text{min}(B_{i+1}) = \text{min} (S\setminus \bigcup\limits_{j=1}^i B_i) $. As a consequence, for any $i$ and $a\in B_i$, $a\leq 2\text{min}B_i $. Block analysis of metric spectra in the study of big Ramsey degree was initially developed by Masul\v{o}vi\'c \cite{Mas}. While our definition differs slightly, at its core, it serves a similar combinatorial purpose.

	\begin{definition}
		Let $S$ be a finite metric spectra with block decomposition $S= \bigcup\limits_{i=1}^k B_i $. We say two distinct blocks $B_i, B_j$ $i<j$, are \textit{independent} if $2 \textnormal{max}B_i < \textnormal{min}B_j $. We say $S$ is simple if every pair of distinct blocks is independent. 
	\end{definition}
	Simple spectra form a very small class of metric spectra. However, independence is necessary to encode metric spaces onto indexed graphs. Moreover, simple spectra satisfy \textit{the four value condition} (see \cite{Mas} for a definition) and hence, $\mathcal{M}_S$ is a \Fra\ class when $S$ is simple. Hence, there is an associated Urysohn space $\mathbb{U}_S$. 
	\begin{lemma}[Theorem 5.5 \cite{Mas}]
		If $S$ is a finite metric spectra with one block, the countable Urysohn space $\mathbb{U}_S$ has finite big Ramsey degrees.
	\end{lemma}
	\begin{proof}
		Important to note in this proof is that metric spaces with spectra in $S$ can be identified with $|S|$-edge-coloured complete graphs. Moreover, $\mathbb{U}_S$ can be identified with the ultrahomogeneous $|S|$-edge-coloured complete graph. As this problem is already solved, we have finite big Ramsey degrees. 
	\end{proof}

	\begin{lemma}\label{siminternal}
		Let $S$ be a simple metric spectra, and let $\mathbb{U}_S $ be the countable Urysohn space. The relation $x\sim y$ if and only if $x=y$ or $d(x,y) \in B_1$ is an equivalence relation. 
	\end{lemma}
	\begin{proof}
		Symmetry and reflexivity is clear. Suppose $x,y,z$ are such that $d(x,y), d(y,z) \in B_1\cup \{0\}$. By the triangle inequality, we have
		\begin{align*}
			d(x,z) &\leq d(x,y) + d(y,z)\\
			&\leq 2\text{max} B_1\\
			&<  \text{min}B_2
		\end{align*}
		Thus, $d(x,z) \in B_1$.
	\end{proof}
	\begin{lemma}\label{simexternal}
		Suppose $x\sim y$ for a pair $x,y\in \mathbb{U}_S$. If $d(x,z) \in B_i$, then $d(y,z) \in B_i$. 
	\end{lemma}
	\begin{proof}
		Again, by triangle inequality and independence, we have
		\begin{align*}
			d(y,z) &\leq d(x,y) + d(x,z)\\
			&\leq \text{max}B_1 + \text{max}B_i\\
			&\leq 2 \text{max}B_i
		\end{align*}
		Also,
		\begin{align*}
			d(y,z) &\geq d(x,z) - d(x,y)\\
			&\geq \text{min} B_i - \text{max}B_1
		\end{align*}
		Assuming $i\neq 1$ (else there would be nothing to check) we get that $d(y,z) \geq \text{max} B_{i-1} $, else $\text{min} B_i \leq 2 \text{max}B_{i-1} $ contradicting independence. Hence, $d(y,z) \in B_i$ as desired. 
	\end{proof}
	
	\begin{theorem}
		If $S$ is a simple spectra. Suppose the block decomposition $S=\bigcup\limits_{i=1}^k B_i$ satisfies the condition $|B_1| = k-1$, and $|B_i| = 1$ for $i> 1$. $\mathbb{U}_S$ admits finite big Ramsey degrees.
	\end{theorem}
	\begin{proof}
		Let $\mathcal{K}$ denote the category of $k-1$-edge-coloured complete graphs. We will create an isomorphism of categories $F:\mathcal{M}_S\rightarrow \KK^\KK  $. We do so as follows:
		\begin{itemize}
			\item We enumerate $S= \{s_1,...,s_{2k-2}\}$ in increasing order.
			\item We take the $\sim$ partition $\{P_i: i\in I\}$.
			\item We endow $P_i$ with a $k-1$-edge coloured complete graph structure with $xR_j y$ if and only if $d(x,y) = s_j$. This is well defined by Lemma \ref{siminternal}.
			\item We endow $I$ with a $k-1$-edge coloured complete graph structure by $i_1 R_j i_2$ if and only if $ \exists x\in P_{i_1}$ and $y\in P_{i_2}$ such that $d(x,y) = s_{k-1+j}$. This is well defined by Lemma \ref{simexternal}.
			\item Define $F((x,d)) = ( (I, \{R_j:j\in\{1,...,k-1\}),l)$, where $l(i) = (P_i,\{R_j:j\in\{1,...,k-1\})$.
		\end{itemize}
		It is clear that this is an isomorphism of category. Any isometric embedding $f:(X,d_X) \rightarrow (Y,d_y)$ lifts uniquely to a map $\hat{f}: F((X,d_X)) \rightarrow F((Y,d_Y))$. Moreover, the process described in $F$ is reversible, and an embedding of $\KK^\KK$ structures can also be identified with a corresponding $\mathcal{M}_S$ structure. The above process is also blind to cardinality, so we can make sense of the structure $F(\mathbb{U}_S)$. Ultrahomogeneity of $\mathbb{U}_S$ gives us that $F(\mathbb{U}_S)$ is ultrahomogeneous. Hence, $F(\mathbb{U}_S)$ can be identified with $(\K,l_\K)$, where $\K$ is the ultrahomogeneous $k-1$-edge coloured graph. \\
		\\
		From this point, the computation of big Ramsey degree is near immediate. We will highlight the argument for the sake of completion. It suffices to consider the embedding definition of big Ramsey degree, where we instead consider a colouring $\chi: \textbf{Emb} (\A,\mathbb{U}_S ) \rightarrow n$, where $\A \in \mathcal{M}_S$ finite. Since $F$ is an isomorphism of category, we can instead consider the corresponding colouring $\hat{\chi}: \textbf{Emb} (F(\A),(\K,l_\K)) \rightarrow n $ given by
		\begin{align*}
			\hat{\chi}(\hat{f}) &= \chi(f)
		\end{align*}
		Since every finite substructure of $\K$ has finite big Ramsey degrees, we have $t=t_\square(F(A), (\K,l_\K)) $ is finite. Hence, we can find a copy of $(\K,l_\K))$ that meets at most $t$ many colour classes with respect to $\hat{\chi}$. Since every copy of $(\K,l_\K)$ corresponds to a self embedding $\textbf{Emb}(\mathbb{U}_S, \mathbb{U}_S) $, we can find a copy $X\in \binom{\mathbb{U}_S}{\mathbb{U}_S}$ such that $|\chi[\textbf{Emb}(\A,X) ]| \leq t$. Hence, $\mathbb{U}_S$ admits finite big Ramsey degree.  
	\end{proof}
	
	The above is highly reminiscent of the Ne\v{s}et\v{r}il-R\"{o}dl partite construction \cite{NesRod}. While the case for a given $S$ is hyper specific, it is quite possible the above can be adapted to serve a wider class of metric spaces. Moreover, the above is entirely combinatorial. 
	
	\section{Concluding Remarks}
	We will conclude with some natural questions that arise. We will start with the one of most significance to the author, which was initially attempted by Vuksanovic. 
	
	\begin{question}
		For every $n$, find a clear description of the basis for linear orders of size $n$ in $\QQ$. 
	\end{question}
	
	It is clear by now a Basis exists. The author has some notes isolating what many of the equivalence relations must look like using transitivity. The current best guess uses ''stable" pairs of Devlin types $(A,B)$, which do not generate/impose any other distinct pair $(C,D)$ must belong to defining the equivalence relation by $\simeq$-chain. It is unclear whether or not all equivalence relations are uniquely determined by a set of stable pairs. 
	
	\begin{question}
		Suppose $(\mathbf{J},l) \in \mathcal{K}^\mathcal{I}$ is such that $\mathbf{J}$ has finite big Ramsey degree in $\mathbf{I}$, and $t_\square ( (\A_i)_{i\in\mathbf{I}}, \K) <\infty$. Does $(\mathbf{J},l)$, where $l(i) = \A_i$, have finite big Ramsey degree in $(\mathbf{I}, l_\mathbf{K})$?
	\end{question}
	We were only able to solve the above using the case $\K=\I$ by using the box Ramsey machinery. There is still much to be done when $\K$ and $\I$ are highly uncorrelated structures. What would be most intriguing is if the above fails. This would likely require $\K$ and $\I$ be chosen to have a strong degree of friction between each other that could be exploited by a colouring. 
	\begin{question}
		Are there nontrivial uncountable structures that admit Ramsey bases for their finite substructures?
	\end{question}
	It was noted by Baumgartner that if $\kappa$ is a Ramsey cardinal, the class Rado proof shows that projections form a Ramsey basis for $[\kappa]^n$ \cite{Baum}. The natural followup is then if there are interesting uncountable structures that admit Ramsey bases. 
	\begin{question}
		Can we express the results here in a completely categorical framework?
	\end{question}
	The category theoretical framework has proven to be beneficial in Ramsey theory on many occasions. The natural question to then ask is if the results here could be lifted to general categories. The key trick used in our proof was to consider unions of structures. There is a way to generalize unions in the category setting, however, it is still unclear that the proof would hold given how abstract and bare-bones the categorical definition is. This leads naturally to the next question. 
	\begin{question}
		Is there a reasonable approximate (or even projective) notion of Ramsey basis? 
	\end{question}
	This is arguably would be easier and more practical to consider than working explicitly in the categorical framework. However, the categorical framework is all-encompassing, and likely would solve the above question. 
	\\
	\\
	Finally, we conclude with the question that initially motivated this work. Namely, whether or not there is a connection between relational expansions and topological dynamics. 
	\begin{question}
		Suppose for every $\A \in \KK$, there is a finite Ramsey basis $\mathcal{B} \subseteq \EE(\A,\K)$. Can the natural actions of $\textnormal{Emb}(\K,\K)$ on the compact spaces $\EE(\A,\K)$ be used to compute a universal proximal flow?
	\end{question}
	
	\section{Acknowledgements}
	The author would like to thank Natasha Dobrinen, Spencer Unger, and Stevo Todorcevic for their comments. They helped greatly with the fine analysis of $\QQ$, which lead to some of the more general results listed here. The author would also like to thank Dragan Masul\v{o}vi\'c, as his motivation was critical to the completion of this project. Finally, the author would like to thank the Fields Institute for Research in Mathematical Sciences for funding the project.

\end{document}